\documentclass[]{interact}

\usepackage{epstopdf}
\usepackage{subfigure}

\usepackage{natbib}
\bibpunct[, ]{(}{)}{;}{a}{}{,}
\renewcommand\bibfont{\fontsize{10}{12}\selectfont}

\theoremstyle{plain}
\newtheorem{theorem}{Theorem}[section]

\newtheorem{proposition}{Proposition}[section]

\theoremstyle{definition}
\newtheorem{definition}{Definition}[section]
\newtheorem{example}{Example}[section]
\theoremstyle{remark}
\newtheorem{remark}{Remark}[section]

\begin{document}
	
	
	\title{Metric basis and dimension of barycentric subdivision of zero divisor graphs}
	
	\author{
		\name{S. Vidya\textsuperscript{a}\thanks{CONTACT G. R. Vadiraja Bhatta . Email: vadiraja.bhatta@manipal.edu}, Sunny Kumar Sharma
        \textsuperscript{b}, Prasanna Poojary\textsuperscript{a}, Omaima Alshanqiti\textsuperscript{c}, and G. R. Vadiraja Bhatta\textsuperscript{a} }
		\affil{\textsuperscript{a} Manipal Institute of Technology, Manipal Academy of Higher Education, Manipal, India; \textsuperscript{b}School of Mathematics, Shri Mata Vaishno Devi University, Katra-182320, Jammu and Kashmir, India;
  \textsuperscript{c}Department of Mathematics, Umm al-qura University, Makkah, Saudi Arabia;
  \textsuperscript{d}Department of Mathematics, Manipal Institute of Technology, Manipal Academy of Higher Education, Manipal, India}}
	
	\maketitle
	
	\begin{abstract}
		Let $R$ be a commutative ring with unity 1, and $ G(V,E)$ be a simple, connected, nontrivial graph. Let $d(a,c)$ be the distance between the vertices $a$ and $c $ in $G$. An undirected zero divisor graph of a ring $R$ is denoted by $\Gamma(R) = (V(\Gamma(R)), E(\Gamma(R)))$, where the vertex set $V(\Gamma(R))$ consists of all the non-zero zero-divisors of $R$, and the edge set $E(\Gamma(R))$ is defined as follows: $E(\Gamma(R)) = $ $\{e = a_1a_2$ $ |$ $
a_1 \cdot a_2 = 0$  $\&$ $  a_1, a_2 \in V(\Gamma(R))\}$. In this article, we consider the zero divisor graph of a group of integers modulo \(n\), denoted as \(\Gamma(\mathbb{Z}_n)\), where \(n=pq\). Here, \(p\) and \(q\) are distinct primes, with \(q > p\). We aim to determine the metric dimension of the barycentric subdivision of the zero divisor graph \(\Gamma(\mathbb{Z}_n)\), denoted by \(dim(BS(\Gamma(\mathbb{Z}_n)))\), and we also prove that \(dim(BS(\Gamma(\mathbb{Z}_n)))\geq q-2\) for every \(n=pq\), where \(p\) and \(q\) are distinct primes and $q>p$.
	\end{abstract}
	

	\section{Introduction}
	
	\noindent Let $G(V, E)$ be a simple graph with a vertex set $V(G)$ and an edge set $E(G)$. The degree of a vertex $v$ is denoted by $d(v)$ and represents the number of edges connected to vertex $v$. A graph is called a regular graph if all of its vertices have the same degree. A path between two vertices $p_1$ and $p_n \in V(G)$ is an ordered sequence of distinct vertices $p_1, p_2,\ldots,p_n$ of $G$, such that $p_{j-1}p_j$ is an edge for $2\leq j\leq n$. The distance between two vertices $v_1$ and $v_2$, denoted by $d(v_1, v_2)$, is the length of the shortest path from $v_1$ to $v_2$ in $G$. A graph $G$ is connected if there is a path between every pair of vertices in $G$; otherwise, it is disconnected. An independent set $S \subseteq V(G)$ is a set in which the subgraph induced by $S$ is completely disconnected. A bipartite graph is a graph in which the vertex set is divided into two disjoint sets, denoted by \( V_1 \) and \( V_2 \), every edge \( e = ab \) in \( E(G) \) connects a vertex \( a \) from one set to a vertex \( b \) from the other set, meaning either \( a \in V_1 \), and \( b \in V_2 \), or \( a \in V_2 \), and \( b \in V_1 \). A bipartite graph is a complete bipartite graph if every vertex in \( V_1 \) is adjacent to every vertex in \( V_2 \). This complete bipartite graph, with \(|V_1| = m\) and \(|V_2| = n\) vertices, is denoted by \( K_{m,n} \). A maximal connected subgraph of a graph \( G = (V, E) \) is a subgraph \( S = (V_S, E_S) \), where \( V_S \subseteq V \), and \( E_S \subseteq E \), such that no additional vertices or edges from \( G \) can be added to \( S \) without losing its connectivity. A tree \( T = (V, E) \) is a connected, undirected graph with no cycles. For each node \( v \in V \) of a tree \( T \), the legs at \( v \) are the bridges, which are paths. The number of legs at $ v$ is represented by $l_v$. For more basics on graph theory, we refer to books (\cite{w}, \cite{bo}).\\

\noindent Beck  (\cite{beck}) first introduced the idea of a zero-divisor graph $\Gamma(R)$ on a commutative ring $R$. His work created a link between ring theory and graph theory. He primarily focused on finitely colourable rings. Later, a new approach to associate zero-divisors with $R$ was presented by Anderson and Livingston (\cite{anderson}). They defined an undirected zero divisor graph of a ring $R$ to be a graph, denoted by $\Gamma(R) = (V(\Gamma(R)), E(\Gamma(R)))$, where the vertex set $V(\Gamma(R))$ consists of all the non-zero zero-divisors of $R$, and the edge set $E(\Gamma(R))$ is described as follows: $E(\Gamma(R)) = $ $\{e = p_1p_2$ $ |$ $
p_1 \cdot p_2 = 0$  $\&$ $  p_1, p_2 \in V(\Gamma(R))\}$. We adopt Anderson and Livingston's approach, considering only non-zero zero-divisors as vertices in graph $\Gamma(R)$. Redmond (\cite{redmond}) proposed an extended definition of a zero-divisor graph that includes non-commutative rings. He demonstrated that for a non-commutative ring $R$, if the sets of left zero divisors and right zero-divisors are identical, then the graph $\Gamma(R)$ is connected. Further, Demeyer et al. (\cite{demyer}) generalised the concept of zero-divisor graphs from rings to semigroups. Additionally, Lucas (\cite{lucas}) analysed the diameter of $\Gamma(R)$. Behboodi (\cite{behboodi}) presented the concept of zero-divisor graphs for modules over commutative rings. For more information on this concept, one can refer to (\cite{sing,vk, ak}).\\

\noindent The idea of metric dimension (MD) was first introduced independently by Slater (\cite{sl}), and Harary and Melter (\cite{ha}). However, Erdos, Harary, and Tutte had already discussed the dimension of graphs (\cite{er}). In their work, Slater focused on the MD of trees, while Harary and Melter discussed the MD of complete graphs, cycles, and bipartite graphs. Slater referred to resolving set and MD as locating set and locating number respectively, while Harary and Melter named these concepts as resolving set and MD. They termed the ordered subset $A=(a_1,a_2,a_3, \ldots,a_v)$ of vertices of graph $G= (V, E)$ as a resolving set if the representation $\delta(p|A)=(d(p,a_1),d(p,a_2), d(p,a_3),...,d(p,a_v))$ for each vertex $p\in V$ is unique. The minimum resolving set is referred to as a metric basis, the cardinality of the metric basis is called the MD. Instead of locating set and locating number we will consider the terms resolving set and MD throughout this article. For recent work related to MD, one can refer to (\cite{alshe, nazeer, anand, sha, vid}).\\
   
\noindent Pirzada and Raja (\cite{pir}) first investigated the MD of a zero-divisor graph. They have examined the MD of $\Gamma(R \times F_s)$, where $F_s$ is a finite field. Additionally, they computed the MD of $\Gamma(R_1\times R_2\times R_3\times \ldots,\times R_r)$;
 where all $R_i$, $1\leq i\leq r$
  are finite commutative rings. Later, Pirzada et.al. (\cite{pi}) determined the MD and upper dimensions of zero-divisor graphs associated with commutative rings. They also studied several properties of the MD of the zero divisor graph. Later Sharma and Bhat (\cite{ss}) analyzed the fault-tolerant MD of a zero divisor graph and also the line graph of a zero divisor graph associated with a commutative ring. \\
 
 \noindent The MD is an important parameter that can considerably improve graph analysis. It has a wide range of applications, including robot navigating (\cite{kh}), combinatorial optimization (\cite{se}), telecommunication (\cite{be}), coast guard Loran (\cite{sl}), SONAR, network discovery and verification (\cite{be}), pharmacological activity and drug discovery (\cite{chr}), Joins in graphs (\cite{se}), coin weighing problems (\cite{so}), geographical routing protocol (\cite{li}), mastermind games (\cite{chv}), image processing and pattern recognition (\cite{be,ra}).
 There are numerous variations of the notion of MD, which include local MD, strong MD, fractional MD, fault-tolerant MD, edge MD, mixed MD, etc. \\
 
 \noindent  An operation that involves splitting an edge into two edges by inserting a new vertex into the interior of the edge is called subdivision an edge. When this operation is applied to a sequence of edges in a graph $G$, it is referred to as subdividing the graph \( G \). The resulting graph is known as the subdivision of graph \( G \). If all edges in a graph are subdivided, it is called a barycentric subdivision of the graph (\cite{koam}). Gross and Yellen (\cite{jl}) have illustrated that a barycentric subdivision of the graph is both bipartite and loopless. Our research will investigate the MD of the barycentric subdivision of the zero divisor graph, which is denoted by $BS(\Gamma(R))$.\\

  \noindent  In the second section, we will cover the basic definitions and some results of MD and zero divisor graphs. Following that, in the third section, we will explore the MD of the barycentric subdivision of $\Gamma(Z_{pq})$, where $p\geq2$ and $q>p$. Finally, we will conclude that the MD and independent MD of the barycentric subdivision of $\Gamma(Z_{pq})$ is greater than or equal to $q-2$, where $p\geq 2$, and $q>p$. Throughout this article, we consider $p$ and $q$ to be two distinct primes.

 \section{Preliminaries}
This section contains the basic definitions and results of MD and zero divisor graphs, which are useful for proving our main result.
 \begin{definition}(\cite{su})
\noindent The neighbourhood (open neighbourhood) of a vertex $p$ in a graph $G=(V, E)$, is denoted by $N(p)$, and is defined as 
$N(p) = \{x \in V(G) | px \in E(G)\}.$ The closed neighborhood of a vertex $p$, denoted by $N[p]$, it is the set $\{p\} \cup
N(p)$.
 \end{definition}
  \begin{definition}(\cite{cht} )
A vertex $a\in V(G)$ resolves a pair of vertices $x$ and $y$, if $d(a,x)\neq d(a,y)$, where $x, y \in V(G)$.
 \end{definition}
 \begin{definition}(\cite{ha})		
\noindent Let  $A=\{a_1, a_2, a_3, \ldots,a_k\}$ be an ordered set of vertices in $G$. The metric coordinate$/$metric code of a vertex $s$ is denoted by $\delta(s|A)$, and defined as $\delta(s|A)=(d(s,a_1),d(s,a_2),d(s,a_3), \ldots,d(s,a_k))$. The set $A$ is a resolving set if {every two different vertices have distinct metric coordinates} with respect to the set $A$.
   \end{definition}   
    \begin{definition}(\cite{ha}) 
  A resolving set with minimum cardinality is called the metric basis for $G$ and its cardinality is termed as the MD of $G$. It is usually denoted by $dim(G).$ 
 \end{definition}
  \begin{definition}(\cite{cht})\label{defn1}
 \noindent A set $A$ in $G$ is an independent resolving set of $G$ if $A$ is both an independent and a resolving set.
 \end{definition}  
 \begin{definition}(\cite{cht}) \label{defn2}
 An independent resolving set with minimum cardinality is called the independent metric basis for $G$ and its cardinality is termed as the independent metric dimension of $G$. It is usually denoted by $idim(G)$. 
 \end{definition}  
\noindent \begin{proposition}(\cite{kh})
Let $G=(V,E)$ be a graph with MD 2 and let $B=\{a,b\}\subset V$ be a metric basis in $G$, then the following are true:
\begin{itemize}
\item There exists a unique shortest path $P$ between the vertices $a$ and $b$
\item The degrees of both $a$ and $b$ can never exceed 3.
\item Every other vertex on $P$ has a degree of at most five. 
\end{itemize}	
\end{proposition}
   
\noindent\begin{proposition} (\cite{kh})
A graph $G=(V, E)$ has MD one $\iff$ $G$  is a path.
\end{proposition}
\noindent \begin{proposition}(\cite{kh} \label{0})
     Let $T = (V, E)$ be a tree that is not a path. If $l_v$ is the number of legs attached to the
vertex v. Then
\[ dim(T ) =  \sum_{ v \in V:l_v>1  } l_v-1\]
 \end{proposition}
\section{ The metric dimension of barycentric subdivision of  $\Gamma(\mathbb{Z}_n)$}
\noindent In this particular section, we will determine the MD of the barycentric subdivision of $\Gamma(\mathbb{Z}_n)$, where $n=pq$, $p$ and $q$ are two distinct primes with $q>p$.

\begin{theorem}
Let $p$ and $q$ be two distinct primes, where $q > p$. Suppose $R=\mathbb{Z}_n$, where $n=pq$, then
\begin{enumerate}
    \item For $p=2$, dim($BS(\Gamma(\mathbb Z_n))
)=q-2$
\item For $p=3$, $dim(BS(\Gamma(\mathbb Z
_n))
)=q-2$
\end{enumerate} 
\end{theorem}
\begin{proof}
\noindent{\textbf{Case 1: $ p=2$ }}\\
\begin{figure}[ht!]
    \centering
    \includegraphics[width=10cm]{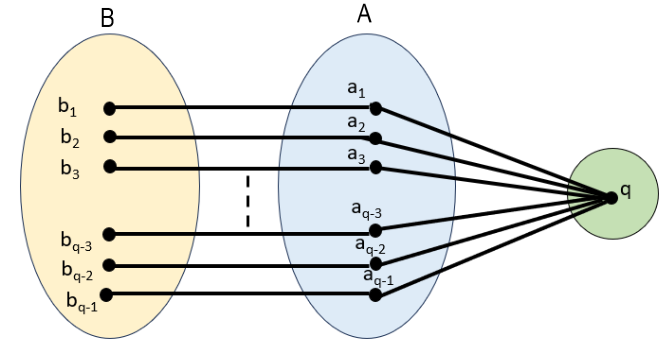}
    \caption{ Barycentric subdivision of zero divisor graph of $\mathbb{Z}_{2q}$}
    \label{Fig:1}
\end{figure}
\noindent Consider a ring \( R = \mathbb{Z}_n \), where \( n = 2q \), and let \( \Gamma(R) \) represent its zero divisor graph. We denote \( BS(\Gamma(R)) \) to be the barycentric subdivision of \( \Gamma(R) \). The graph \( BS(\Gamma(R)) \) consists of \( 2q - 1 \) vertices, and \( 2q - 2 \) edges, as illustrated in Figure \ref{Fig:1}. \noindent The graph \( BS(\Gamma(\mathbb{Z}_{2q})) \) is a tree. From Proposition \ref{0}, we can conclude that the \( dim(BS(\Gamma(\mathbb{Z}_{2q}))) \) is \( q - 2 \). This concludes the proof for case 1.\\

\noindent\textbf{Case 2: $ p=3$ }\\
 \begin{figure}[ht!]
    \centering
    \includegraphics[width=12cm]{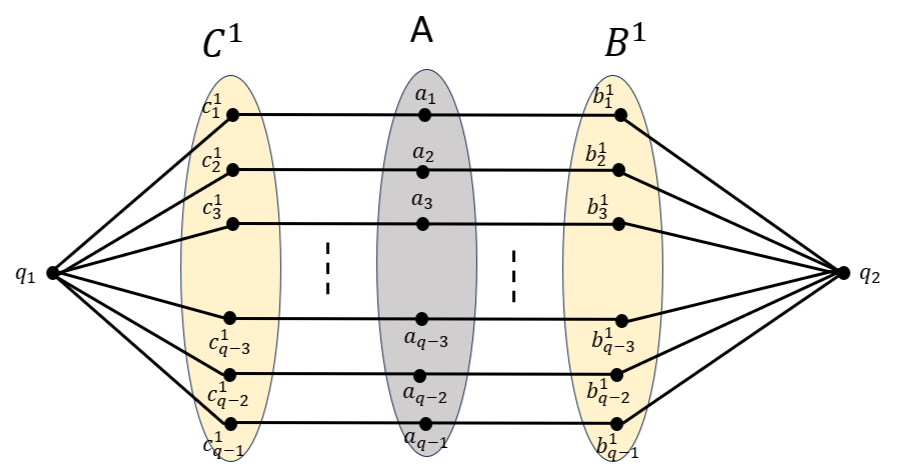}
    \caption{Barycentric Subdivision of Zero Divisor Graph of $\mathbb{Z}_{3q}$}
    \label{fig:2}
\end{figure}
\noindent Consider a ring $R =\mathbb{Z}_n$, where $n=3q$ (where $q>p$), and let $\Gamma(R)$ represent its zero divisor graph. We denote the $BS(\Gamma(R))$ as the barycentric subdivision of $\Gamma(R)$. The graph $BS(\Gamma(R))$ consists of $3q-1$ vertices, and $4(q-1)$ edges. In particular, we have partioned the vertex set of $BS(\Gamma(R))$ into four sets: $A=\{a_1,a_2,\ldots,a_{q-1}\}$, $B^1=\{b^1_1,b^1_2,\ldots,b^1_{q-1}\}$, $C^1=\{c^1_1,c^1_2,c^1_3,\ldots,c^1_{q-1} \}$, and $Q=\{q_1,q_2\}$. These sets are illustrated in Figure \ref{fig:2}.\\

\noindent Let us consider a set $E=\{a_1,b^1_2,b^1_3,\ldots, b^1_{q-3},b^1_{q-2}\}$. Then, to show that $E$  is a resolving set of $BS(\Gamma(R))$, we need to assign unique metric coordinates for every vertex of $V(BS(\Gamma(R)))-E$ with respect to set $E$. \\
 The metric coordinates for the vertices $\{a_ \varrho:2\leq \noindent \varrho\leq q-1\}$ are given below
 \begin{equation*}
               \delta(a_{\varrho}|E)=
                   \begin{cases}
                   (4_{\{1\}},1_{\{2 \}},3_{\{3\}},\ldots,3_{\{q-3\}},3_{\{q-2\}}) & \textit{if} \quad \varrho =2 ;\\
                   (4_{\{1\}},3_{\{2\}},\ldots,3_{\{\varrho -1\}},1_{\{ \varrho \}},3_{\{\varrho +1\}},\ldots,3_{\{q-3\}},3_{\{q-2\}}) & \textit{if} \quad 3 \leq \varrho \leq q-3  ;\\
                    (4_{\{1\}},3_{\{2\}},\ldots,3_{\{q-3\}},1_{\{q-2\}}) & \textit{if} \quad \varrho =q-2  ;\\
                    (4_{\{1\}},3_{\{2\}},\ldots,3_{\{q-3\}},3_{\{q-2\}}) & \textit{if}\quad \varrho= q-1.\\
                    \end{cases}
                     \end{equation*}
                     \noindent The metric coordinates for the vertices 
                     \noindent$b^1_1$ and $b^1_{q-1}$  are given below
                     \begin{center}
                          ${\delta(b^1_{1}|E)= (1_{\{ 1\}},2_{\{2\}},\ldots,2_{\{q-2\}}) }$\\
                   \noindent${\delta(b^1_{q-1}|E)= (3_{\{1\}},2_{\{2\}},\ldots,2_{\{q-2\}})}$ \\
                     \end{center}
                     \noindent The metric coordinates for the vertices 
                     $\{q_1,q_2\}$ are given below
                     \begin{center}
\noindent $\delta(q_1|E)=(2_{\{1\}},3_{\{2\}},3_{\{3\}},\ldots,3_{\{q-3\}},3_{\{q-2\}})$\\
\noindent $ \delta(q_2|E)=(2_{\{1\}},1_{\{2\}},1_{\{3\}},\ldots,1_{\{q-3\}},1_{\{q-2\}})$\\
\end{center}
\noindent The metric coordinates for the vertices 
                     $\{c^{1}_{\varrho}:1\leq \varrho\leq q-1\}$ are given below
                        \begin{equation*}
                        \delta({c^{1}_{\varrho}}|E)=
                   \begin{cases}
                   (1_{\{ 1 \}},4_{\{2\}},\ldots,4_{\{q-2\}},4_{\{q-2\}}) &\textit{if}\quad \varrho=1;\\
                   (3_{\{1\}},2_{\{2 \}},4_{\{3\}},\ldots,4_{\{q-3\}},4_{\{q-2\}}) &\textit{if}\quad \varrho=2 ;\\
                    (3_{\{1\}},4_{\{2\}},\ldots,4_{\{\varrho -1\}},2_{\{ \varrho \}},4_{\{\varrho +1\}},\ldots,4_{\{q-2\}}) & \textit{if}\quad 3 \leq  \varrho $ $\leq q-3  ;\\
                    (3_{\{1\}},4_{\{2\}},\ldots,4_{\{q-3\}},2_{\{q-2\}}) &\textit{if}\quad \varrho=q-2 ;\\
                    (3_{\{1\}},4_{\{2\}},4_{\{3\}},\ldots,4_{\{q-3\}},4_{\{q-2\}}) &\textit{if} \quad \varrho= q-1.\\
                    \end{cases}
			\end{equation*}
 \noindent Based on the  above observation, all vertices of $BS(\Gamma(R))$ have a unique metric coordinate with respect to the set $E$. It follows that the $dim(BS(\Gamma(R))) \leq q-2$. Next, we need to show that $dim(BS(\Gamma(R)))$ is greater than or equal to $q-2$. For this, we assume that $dim(BS(\Gamma(R)))$ is less than or equal to $q-3$. However, this leads to a contradiction, This can be illustrated below.\\
 
\noindent Let $F$ be a set, which contains any $q-3$ vertices of $D$ (where $D= A\cup B^1 \cup C^1  \cup Q $). Since set $A$ has $q-1$ vertices, and satisfies the condition $N[a_r]\cap  N[a_s]=\emptyset $ (where $a_s$, $a_r \in A$, and $1\leq r \neq s \leq q-1$), there will exist at least two vertices in $A$, denoted as $a_{u} $ and $a_{v} $ (where $1\leq u, v \leq q-1, $ $u \neq v$), and those two vertices not adjacent to any vertex in $F$, and also not belonging to $F$. From Remark \ref{11} vertices $a_u $ and $a_v $ have the same metric coordinates with respect to set $F$, i.e.,      
 $\delta(a_{u}|F) =\delta(a_{v}|F)$.
 This indicates that $dim(BS(\Gamma(R)))\geq q-2$. Thus $dim(BS(\Gamma(R)))= q-2.$ 

\end{proof}
\begin{remark}\label{11}
 From Figure \ref{fig:2}, it is evident that, if $d(a_i,x)\neq 1$ and  $d(a_j,x) \neq 1$ (where $1\leq i, j \leq q-1$, $i \neq j$ and $x\in V(BS(\Gamma(\mathbb{Z}_{3q})))\backslash\{a_i,a_j\}$, then $d(a_i,x)=d(a_j,x)$.   
  \end{remark}

\begin{theorem}\label{b}
Let $p$ and $q$ be two distinct odd primes, where $p\geq 5$, and $q\geq 2p -1$. Suppose $R=\mathbb{Z}_n$, where $n=pq$, then,
      dim($BS(\Gamma(R)))
)=q-2$.
\end{theorem}
\begin{figure}[ht!]
    \centering
    \includegraphics[width=14cm]{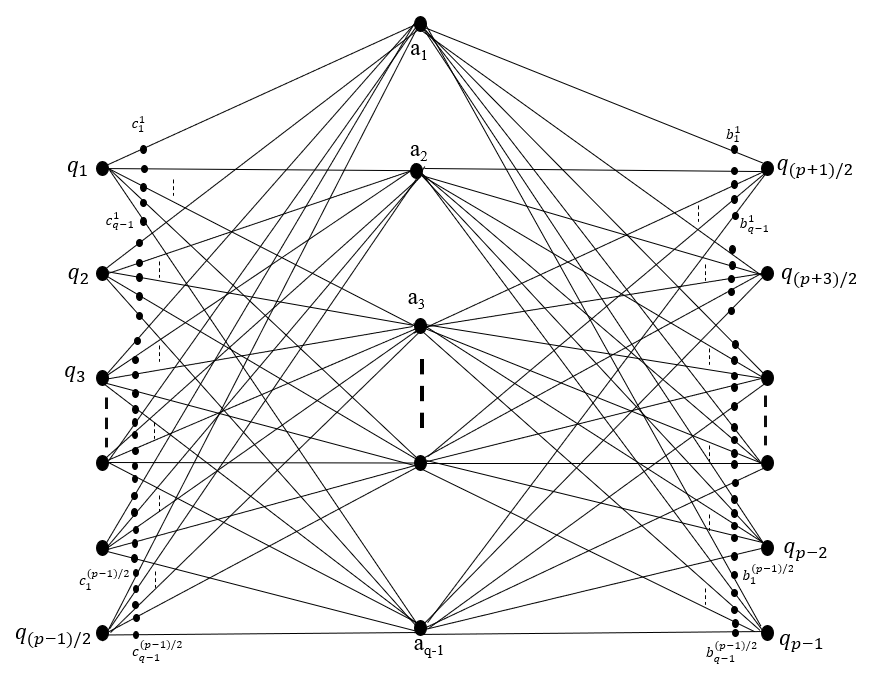}
    \caption{Barycentric Subdivision of Zero Divisor Graph of $\mathbb{Z}_{pq}$}
    \label{fig:4}
\end{figure}
\begin{proof}
  

\noindent Consider the ring \( R = \mathbb{Z}_n \), where \( n = pq \), and \( q \geq 2p - 1 \). Let \( \Gamma(R) \) represent the zero divisor graph of \( R \). We denote \( BS(\Gamma(R)) \) as the barycentric subdivision of \( \Gamma(R) \). The graph \( BS(\Gamma(R)) \) consists of \( pq - 1 \) vertices, and \( 2(p - 1)(q - 1) \) edges. In particular, we partitioned the vertex set of $BS(\Gamma(R))$ into 4 sets namely $A=\{a_1,a_2,\ldots, a_{q-1}\}$,  $B=\{b^\nu_i\}$, $C=\{c^\nu_i\}$, where $1\leq \nu\leq \frac{p-1}{2}$, and $1\leq i\leq {q-1}$ ($B^\nu=\{b^\nu_1,b^\nu_2,\ldots,b^\nu_{q-1}\}$, where $1\leq \nu \leq \frac{p-1}{2}$ are the subsets of $B$, and $C^\nu=\{c^\nu_1,c^\nu_2,\ldots,c^\nu_{q-1}\}$, where $1\leq \nu \leq \frac{p-1}{2}$  are the subsets of $C$), and $Q=\{q_1,q_2, \ldots, q_{p-1}\}$. These sets are illustrated in Figure \ref{fig:4}.\\

 \noindent Let us consider a set $E=\{   a _1, c^1_2,c^1_3,c^2_4,c^2_5, \ldots,c^\frac{p-1}{2}_{p-1},c^\frac{p-1}{2}_{p}, b^1_{p+1}, b^1_{p+2}, \ldots,b^\frac{p-5}{2}_{2p-6},b^\frac{p-5}{2}_{2p-5},b^\frac{p-3}{2}_{2p-4}, \ldots ,b^\frac{p-3}{2}_{q-2}\}$. 
Then, to show that $E$  is a resolving set of $BS(\Gamma(R))$, we need to assign unique metric coordinates for every vertex of $V(BS(\Gamma(R)))-E$ with respect to set $E$. \\
 The metric coordinates for the vertices $\{a_ \varrho:2\leq \noindent \varrho\leq q-1\}$ are given below
 \begin{equation*}
               \delta(a_{\varrho}|E)=
                   \begin{cases}
                   (4_{\{1\}},1_{\{ 2 \}},3_{\{3\}},\ldots,3_{\{q-3\}},3_{\{q-2\}}) &\textit{if}\quad \varrho= 2;\\
                  
                    (4_{\{1\}},3_{\{2\}},\ldots,3_{\{\varrho -1\}},1_{\{ \varrho \}},3_{\{\varrho +1\}},\ldots,3_{\{q-2\}}) &\textit{if}\quad 3\leq \varrho $ $\leq q-3  ;\\
                    (4_{\{1\}},3_{\{2\}},\ldots,3_{\{q-3\}},1_{\{q-2\}}) &\textit{if}\quad \varrho= q-2;\\
                  
                    (4_{\{1\}},3_{\{2\}},3_{\{3\}},\ldots,3_{\{q-3\}},3_{\{q-2\}}) &\textit{if}\quad \varrho= q-1.\\
                \end{cases}
                     \end{equation*}
                     \noindent The metric coordinates for the vertices 
                     $\{b^\nu_\varrho:1\leq \nu\leq \frac{p-5}{2} $ $ \& $  $1\leq \varrho\leq q-1 \}$ are given below
                        \begin{equation*}
                        \delta(b^{\nu}_{\varrho}|E)=
                   \begin{cases}
                   (1_{\{1\}},4_{\{2\}},\ldots,4_{\{2\nu+p-2\}},2_{\{p+2\nu-1\}},2_{\{p+2\nu\}},4_{\{p+2\nu+1\}},\ldots,4_{\{q-2\}})\\ \qquad \qquad \qquad \qquad \qquad \qquad \qquad \qquad \qquad  \qquad \qquad\qquad\qquad\quad\textit{if}\quad\varrho= 1;\\
                    (3_{\{1\}},2_{\{2 \}},4_{\{3\}},\ldots,4_{\{2\nu+p-2\}},2_{\{p+2\nu-1\}},2_{\{p+2\nu\}},4_{\{p+2\nu+1\}},\ldots,4_{\{q-2\}})\\
\qquad \qquad \qquad \qquad \qquad \qquad \qquad \qquad \qquad  \qquad \qquad\qquad\qquad\quad\textit{if}\quad   \varrho= 2;\\
(3_{\{1\}},4_{\{2\}}, \ldots,4_{\{\varrho-1 \}},2_{\{ \varrho \}},4_{\{\varrho +1\}},\ldots,4_{\{2\nu+p-2\}},2_{\{p+2\nu-1\}},2_{\{p+2\nu\}},4_{\{p+2\nu+1\}},\ldots,4_{\{q-2\}})\\
\qquad \qquad \qquad \qquad \qquad \qquad \qquad \qquad \qquad  \qquad \qquad\qquad\qquad\quad\textit{if} \quad  3 \leq  \varrho  \leq 2\nu+p-3 ;\\

                     (3_{\{1\}},4_{\{2\}}, \ldots,4_{\{\varrho-1 \}},2_{\{ 2\nu+p-2 \}},2_{\{p+2\nu-1\}},2_{\{p+2\nu\}},4_{\{p+2\nu+1\}},\ldots,4_{\{q-2\}})\\
\qquad \qquad \qquad \qquad \qquad \qquad \qquad \qquad \qquad  \qquad \qquad\qquad\qquad\quad\textit{if}\quad \varrho = 2\nu+p-2 ;\\
                     (3_{\{1\}},4_{\{2\}}, \ldots,4_{\{2\nu+p-2\}},2_{\{p+2\nu-1\}},2_{\{p+2\nu\}},2_{\{p+2\nu+1\}}, 4_{\{p+2\nu+2\}}, \ldots,,4_{\{q-2\}})\\
\qquad \qquad \qquad \qquad \qquad \qquad \qquad \qquad \qquad  \qquad \qquad\qquad\qquad\quad\textit{if}\quad\varrho= p+2\nu+1 ;\\
                    (3_{\{1\}},4_{\{2\}}, \ldots,4_{\{2\nu+p-2\}},2_{\{p+2\nu-1\}},2_{\{p+2\nu\}},4_{\{p+2\nu+1\}},\ldots,4_{\{\varrho-1 \}},2_{\{ \varrho \}},4_{\{\varrho +1\}},\ldots,4_{\{q-2\}})
\\
\qquad \qquad \qquad \qquad \qquad \qquad \qquad \qquad \qquad  \qquad \qquad\qquad\qquad\quad \textit{if} \quad p+2\nu+2 \leq  \varrho  \leq q-3  ;\\
                     (3_{\{1\}},4_{\{2\}}, \ldots,4_{\{2\nu+p-2\}},2_{\{p+2\nu-1\}},2_{\{p+2\nu\}},4_{\{p+2\nu+1\}},\ldots, 4_{\{q-3\}},2_{\{q-2\}})\\
\qquad \qquad \qquad \qquad \qquad \qquad \qquad \qquad \qquad  \qquad \qquad\qquad\qquad\quad\textit{if}\quad   \varrho = q-2  ;\\
                    (3_{\{1\}},4_{\{2\}}, \ldots,4_{\{2\nu+p-2\}},2_{\{p-1+2\nu\}},2_{\{p+2\nu\}},4_{\{p+2\nu+1\}},\ldots,4_{\{q-2\}}) \\
\qquad \qquad \qquad \qquad \qquad \qquad \qquad \qquad \qquad  \qquad \qquad\qquad\qquad\quad \textit{if}\quad\varrho= q-1.\\
                      \end{cases}
                    \end{equation*}
                    \noindent The metric coordinates for the vertices 
                     $\{b^\nu_\varrho: \nu =\frac{p-3}{2} $ $\&$  $1\leq \varrho\leq q-1 \}$ are given below
                        \begin{equation*}
                        \delta(b^{\nu}_{\varrho}|E)=
                   \begin{cases}
                   (1_{\{1\}},4_{\{2\}},\ldots,4_{\{2p-5 \}},2_{\{2p-4 \}}, \ldots,2_{\{q-3\}},2_{\{q-2\}})\\
\qquad \qquad \qquad \qquad \qquad \qquad \qquad \qquad \qquad  \qquad \qquad\qquad\qquad\quad\textit{if}\quad\varrho= 1;\\
                    (3_{\{1\}},2_{\{2\}},4_{\{3\}}, \ldots, 4_{\{2p-5 \}},2_{\{2p-4 \}}, \ldots,2_{\{q-3\}},2_{\{q-2\}}) 
                    \\
\qquad \qquad \qquad \qquad \qquad \qquad \qquad \qquad \qquad  \qquad \qquad\qquad\qquad\quad\textit{if}\quad\varrho= 2;\\ 
                   (3_{\{1\}},4_{\{2\}}, \ldots,4_{\{\varrho-1\}},2_{\{ \varrho\}},4_{\{\varrho +1\}}, \ldots, 4_{\{2p-5 \}},2_{\{2p-4 \}}, \ldots,2_{\{q-3\}},2_{\{q-2\}})\\
\qquad \qquad \qquad \qquad \qquad \qquad \qquad \qquad \qquad  \qquad \qquad\qquad\qquad\quad\textit{if}\quad   3 \leq  \varrho  \leq 2p-6  ;\\
                      (3_{\{1\}},4_{\{2\}}, \ldots, 4_{\{2p-6 \}} 2_{\{2p-5 \}},2_{\{2p-4 \}}, \ldots,2_{\{q-3\}},2_{\{q-2\}})\\
\qquad \qquad \qquad \qquad \qquad \qquad \qquad \qquad \qquad  \qquad \qquad\qquad\qquad\quad\textit{if}\quad   \varrho= 2p-5;\\
                    (3_{\{1\}},4_{\{2\}}, \ldots, 4_{\{2p-5\}},2_{\{2p-4\}}, \ldots,2_{\{q-3\}},2_{\{q-2\}})\\ \qquad \qquad \qquad \qquad \qquad \qquad \qquad \qquad \qquad  \qquad \qquad\qquad\qquad\quad\textit{if}\quad\varrho= q-1.\\
                     \end{cases}
                    \end{equation*}

                    \noindent The metric coordinates for the vertices 
                     $\{b^\nu_\varrho: \nu =\frac{p-1}{2} $ $\&$  $1\leq \varrho\leq q-1 \}$ are given below
                        \begin{equation*}
                        \delta(b^{\nu}_{\varrho}|E)=
                   \begin{cases}
                   (1_{\{1\}},4_{\{2\}},\ldots,4_{\{q-3\}},4_{\{q-2\}})&\textit{if}\quad\varrho= 1;\\
                   (3_{\{1\}},2_{\{ 2\}},4_{\{3\}}, \ldots,4_{\{q-3\}},4_{\{q-2\}})&\textit{if}\quad   \varrho= 2;\\
                   (3_{\{1\}},4_{\{2\}}, \ldots,4_{\{\varrho-1\}},2_{\{ \varrho\}},4_{\{\varrho +1\}},\ldots,4_{\{q-3\}},4_{\{q-2\}})&\textit{if}\quad  3\leq  \varrho  \leq q-3  ;\\
                      (3_{\{1\}},4_{\{2\}},\ldots,4_{\{q-3\}},2_{\{q-2\}})& \textit{if}\quad \varrho = q-2  ;\\
                (3_{\{1\}},4_{\{2\}}, \ldots,4_{\{q-3\}},4_{\{q-2\}})&\textit{if}\quad\varrho= q-1.\\
                      \end{cases}
                    \end{equation*}
                     \noindent The metric coordinates for the vertices 
                     $\{c^1_\varrho: 1\leq \varrho\leq q-1 \}$ are given below
                        \begin{equation*}
                        \delta(c^{1}_{\varrho}|E)=
                   \begin{cases}
                   (1_{\{1\}},2_{\{2\}},2_{\{3\}},4_{\{4\}},\ldots,4_{\{q-2\}})&\textit{if}\quad\varrho= 1;\\
                     (3_{\{1\}},2_{\{2\}}, 2_{\{3 \}},2_{\{ 4 \}}, 4_{\{ 5 \}},\ldots,4_{\{q-2\}})&\textit{if}\quad  \varrho= 4;\\
                     (3_{\{1\}},2_{\{2\}}, 2_{\{3 \}},4_{\{ 4 \}}, \ldots, 4_{\{\varrho-1 \}},2_{\{\varrho \}}, 4_{\{\varrho+1 \}},\ldots,4_{\{q-2\}})&\textit{if}\quad 5 \leq  \varrho  \leq q-3  ;\\
                     (3_{\{1\}},2_{\{2\}}, 2_{\{3 \}},4_{\{ 4 \}},\ldots, 4_{\{q-3\}},2_{\{q-2\}})&\textit{if}\quad \varrho= q-2;\\
                    (3_{\{1\}},2_{\{2\}},,2_{\{3\}}, 4_{\{4\}},\ldots,4_{\{q-2\}}) &\textit{if}\quad\varrho= q-1.\\\\
                    
                    \end{cases}
                    \end{equation*}

                    \noindent The metric coordinates for the vertices 
                     $\{c^2_\varrho: 1\leq \varrho\leq q-1 \}$ are given below
                        \begin{equation*}
                        \delta(c^{2}_{\varrho}|E)=
                   \begin{cases}
                   (1_{\{1\}},4_{\{2\}},4_{\{3\}},2_{\{4\}}, 2_{\{5\}}, 4_{\{6\}}, \ldots,4_{\{q-2\}})&\textit{if}\quad\varrho= 1;\\
                   (3_{\{1\}},2_{\{2\}},4_{\{3\}},2_{\{4\}}, 2_{\{5\}}, 4_{\{6\}},\ldots,4_{\{q-2\}})&\textit{if}\quad\varrho= 2;\\
                     (3_{\{1\}},4_{\{2\}},2_{\{3\}},2_{\{4\}}, 2_{\{5\}}, 4_{\{6\}},\ldots,4_{\{q-2\}})&\textit{if}\quad\varrho= 3;\\
                     (3_{\{1\}},4_{\{2\}},4_{\{3\}},2_{\{4\}}, 2_{\{5\}}, 2_{\{6\}}, 4_{\{7\}}, \ldots,4_{\{q-2\}})&\textit{if}\quad\varrho= 6;\\
                     (3_{\{1\}},4_{\{2\}}, 4_{\{3 \}},2_{\{ 4 \}}, 2_{\{ 5 \}}, 4_{\{ 6 \}}, \ldots, 4_{\{\varrho-1 \}},2_{\{\varrho \}}, 4_{\{\varrho+1 \}},\ldots,4_{\{q-2\}})&\textit{if}\quad 7 \leq  \varrho  \leq q-3  ;\\
                      (3_{\{1\}},4_{\{2\}},4_{\{3\}},2_{\{4\}}, 2_{\{5\}}, 4_{\{6\}}, \ldots, 4_{\{q-3\}} ,2_{\{q-2\}})&\textit{if}\quad\varrho= q-2;\\
                    (3_{\{1\}},4_{\{2\}},4_{\{3\}},2_{\{4\}}, 2_{\{5\}}, 4_{\{6\}},\ldots,4_{\{q-3\}} ,4_{\{q-2\}})&\textit{if}\quad\varrho= q-1.\\
                    \end{cases}
                    \end{equation*}
                 \noindent The metric coordinates for the vertices 
                     $\{c^\nu_\varrho:3\leq \nu\leq \frac{p-1}{2}$  $\&$  $1\leq \varrho\leq q-1 \}$ are given below
                        \begin{equation*}
                        \delta(c^{\nu}_{\varrho}|E)=
                   \begin{cases}
                   (1_{\{1\}},4_{\{2\}},\ldots, 4_{\{2\nu-1\}},2_{\{2\nu\}},2_{\{2\nu+1\}},4_{\{2\nu+2\}},\ldots,4_{\{q-2\}})\\
                   \qquad \qquad \qquad \qquad \qquad \qquad \qquad \qquad \qquad  \qquad \qquad\qquad\qquad\quad\textit{if}\quad\varrho= 1;\\
                     (3_{\{1\}},2_{\{ 2 \}},4_{\{3\}}, \ldots, 4_{\{2\nu-1\}},2_{\{2\nu\}},2_{\{2\nu+1\}},4_{\{2\nu+2\}},\ldots,4_{\{q-2\}})\\
                     \qquad \qquad \qquad \qquad \qquad \qquad \qquad \qquad \qquad  \qquad \qquad\qquad\qquad\quad\textit{if}\quad \varrho= 2;\\
                   (3_{\{1\}},4_{\{2\}}, \ldots,4_{\{\varrho-1 \}},2_{\{ \varrho \}},4_{\{\varrho +1\}}, \ldots, 4_{\{2\nu-1\}},2_{\{2\nu\}},2_{\{2\nu+1\}},4_{\{2\nu+2\}},\ldots,4_{\{q-2\}})\\
                   \qquad \qquad \qquad \qquad \qquad \qquad \qquad \qquad \qquad  \qquad \qquad\qquad\qquad\quad\textit{if}\quad  3 \leq  \varrho  \leq 2\nu-2  ;\\
                      (3_{\{1\}},4_{\{2\}}, \ldots,4_{\{2\nu-2 \}},2_{\{ 2\nu-1 \}},2_{\{2\nu\}},2_{\{2\nu+1\}},4_{\{2\nu+2\}},\ldots,4_{\{q-2\}})\\
                      \qquad \qquad \qquad \qquad \qquad \qquad \qquad \qquad \qquad  \qquad \qquad\qquad\qquad\quad\textit{if}\quad\varrho= 2\nu-1 ;\\
                     (3_{\{1\}},4_{\{2\}}, \ldots,4_{\{2\nu-1\}},2_{\{2\nu\}},2_{\{2\nu+1\}},2_{\{ 2\nu+2 \}}, 4_{\{2\nu+3 \}},\ldots,4_{\{q-2\}})\\
                     \qquad \qquad \qquad \qquad \qquad \qquad \qquad \qquad \qquad  \qquad \qquad\qquad\qquad\quad\textit{if}\quad  \varrho= 2\nu+2 ;\\
                    (3_{\{1\}},4_{\{2\}}, \ldots,4_{\{2\nu-1\}},2_{\{2\nu\}},2_{\{2\nu+1\}},4_{\{2\nu+2\}},\ldots,4_{\{\varrho-1 \}},2_{\{ \varrho \}},4_{\{\varrho +1\}},\ldots,4_{\{q-2\}})\\
                    \qquad \qquad \qquad \qquad \qquad \qquad \qquad \qquad \qquad  \qquad \qquad\qquad\qquad\quad\textit{if}\quad  2\nu+3 \leq  \varrho  \leq q-3  ;\\
                     (3_{\{1\}},4_{\{2\}}, \ldots,4_{\{2\nu-1\}},2_{\{2\nu\}},2_{\{2\nu+1\}},4_{\{2\nu+2\}},\ldots,4_{\{q-3 \}},2_{\{ q-2 \}})\\
                     \qquad \qquad \qquad \qquad \qquad \qquad \qquad \qquad \qquad  \qquad \qquad\qquad\qquad\quad\textit{if}\quad  \varrho=  q-2  ;\\
                    (3_{\{1\}},4_{\{2\}}, \ldots,,4_{\{2\nu-1\}},2_{\{2\nu\}},2_{\{2\nu+1\}},4_{\{2\nu+2\}},\ldots,4_{\{q-2\}})\\
                    \qquad \qquad \qquad \qquad \qquad \qquad \qquad \qquad \qquad  \qquad \qquad\qquad\qquad\quad\textit{if}\quad\varrho= q-1.\\
                    \end{cases}
                    \end{equation*}
                    \noindent The metric coordinates for the vertices 
                     $\{q_\nu:1\leq \nu\leq {p-1} \}$ are given below
                        \begin{equation*}
                        \delta(q_{\nu}|E)=
                   \begin{cases}
                   (2_{\{1\}},1_{\{2\}}, 1_{\{3\}}, 3_{\{4\}},\ldots,3_{\{q-2\}})&\textit{if}\quad  \nu=1 ;\\
                   (2_{\{1\}},3_{\{2\}}, \ldots,  3_{\{2\nu-1\}},1_{\{2\nu\}},1_{\{2\nu+1\}},3_{\{2\nu+2\}},\ldots,3_{\{q-2\}})&\textit{if}\quad  2 \leq \nu  \leq p-3 ;\\
                     (2_{\{1\}},3_{\{2\}}, \ldots, 3_{\{2\nu-1\}},1_{\{2\nu\}},1_{\{2\nu+1\}},\ldots,1_{\{q-2\}})&\textit{if}\quad  \nu= p-2 ;\\
                     (2_{\{1\}},3_{\{2\}}, \ldots, 3_{\{q-2\}})&\textit{if}\quad  \nu= p-1.\\
                     \end{cases}
                     \end{equation*}

 \noindent Based on the above listed metric coordinates, all vertices of $BS(\Gamma(R))$ have unique metric coordinates with respect to set $E$. It follows that $dim(BS(\Gamma(R)))\leq q-2$.
We need to show that \( \text{dim}(BS(\Gamma(R))) \geq q-2 \). To do this, we assume that \( \text{dim}(BS(\Gamma(R))) \leq q-3 \). However, this assumption leads to a contradiction, which can be illustrated below.\\

\noindent Let $F$ be a set, which contains any $q-3$ vertices of $D$ (where $D= A\cup B \cup C  \cup Q $). Since set $A$ has $q-1$ vertices, and satisfies the condition $N[a_r]\cap  N[a_s]=\emptyset $ (where $1\leq r, s \leq q-1$, $r\neq s$), there will exist at least two vertices in $A$, denoted as $a_u $, and $a_v $ (where $1\leq u, v \leq q-1, $ $u \neq v$), and those two vertices not adjacent to any vertex of $F$, and also not belonging to set $F$. From Remark \ref{22} those two vertices $a_u $ and $a_v $ have the same metric coordinates with respect to set $F$, i.e.,      
 $\delta(a_{u}|F) =\delta(a_{v}|F)$.
 This indicates that $dim(BS(\Gamma(R)))\geq q-2$. Thus $dim(BS(\Gamma(R)))= q-2.$ 
\end{proof}
\begin{remark}\label{22}
  From Figure \ref{fig:4}, it is evident that, if $d(a_i,x)\neq 1$ and $d(a_j,x) \neq 1$, where $1\leq i, j \leq q-1$, $i\neq j$, and $x\in V(BS(\Gamma(\mathbb{Z}_{pq})))\backslash\{a_i,a_j\}$, then $d(a_i,x)=d(a_j,x)$.     
 \end{remark}
 \begin{remark} \label{last} \quad \\
    \begin{itemize}
        \item From Figure \ref{fig:4}, it is clear that among the \( p-1 \) vertices of the set \( Q \), exactly \( \frac{p-1}{2} \) vertices are adjacent only to the vertices of the set \( B \), and not to any vertices in other sets. Similarly, exactly \( \frac{p-1}{2} \) vertices are adjacent only to vertices of the set \( C \), and not to any vertices in other sets.
    \item From Figure \ref{fig:4}, it is evident that, if $d(q_i,a)\neq 1$ and $d(q_j,a) \neq 1$, where $1\leq i, j \leq p-1$, $i\neq j$, and $a\in V(BS(\Gamma(\mathbb{Z}_{pq})))\backslash\{q_i,q_j\}$, then $d(q_i,a)=d(q_j,a)$.\\     
\end{itemize}
\end{remark}

\begin{theorem}
Let $p$ and $q$ be two distinct odd primes, where $p\geq 5$ and $p+1<q<2p-1$. Suppose $R=\mathbb{Z}_n$, where $n=pq$, then
$dim(BS(\Gamma(R))) > q-2$. 
 \end{theorem}
    \begin{proof}
\noindent Consider the ring $R = \mathbb{Z}_{pq}$, where $p \geq 5$ and $p + 1 < q < 2p - 1$. Let $\Gamma(R)$ represent the zero divisor graph of this ring. We denote $BS(\Gamma(R))$ as the barycentric subdivision of $\Gamma(R)$. The graph $BS(\Gamma(R))$ consists of $pq-1$ vertices, and $2(p-1)(q-1)$ edges. In particular, we partitioned the vertex set of $BS(\Gamma(R))$ into 4 sets namely $A=\{a_1,a_2,\ldots, a_{q-1}\}$,  $B=\{b^\nu_i\}$, $C=\{c^\nu_i\}$, where $1\leq \nu\leq \frac{p-1}{2}$, and $1\leq i\leq {q-1}$ ($B^\nu=\{b^\nu_1,b^\nu_2,\ldots,b^\nu_{q-1}\}$, $C^\nu=\{c^\nu_1,c^\nu_2,\ldots,c^\nu_{q-1}\}$, where $1\leq \nu \leq \frac{p-1}{2}$ are the subsets of $B$ and $C$, respectively), and $Q=\{q_1,q_2, \ldots, q_{p-1}\}$. These sets are illustrated in Figure \ref{fig:4}. \\

\noindent Now we want to prove that dim($BS(\Gamma(R)) > q-2$ when $p+1< q<2p-1$. Let us assume that $W$ is a resolving set of $BS(\Gamma(R))$, and $|W|=q-2$. Then we get the following contradictions. \\

\begin{enumerate}
\item If $q-2$ vertices of the set $W$ are chosen from the set $A$, then, according to Remark \ref{last}, all the vertices of $Q$ have the same metric coordinates with respect to $W$, i.e., $\delta(q_{u}|W) =\delta(q_{v}|W)$, where $q_{u},q_{v} \in Q.$
\item If $q-2$ vertices of the set $W$ are chosen from $B$ (or $C$), then, according to Remark \ref{last}, there will exist at least two vertices of $Q$ that are not adjacent to any vertices of $W$, say $q_{u}$ and $q_{v}$, and those two have the same metric coordinates with respect to $W$ i.e., $\delta(q_{u}|W) =\delta(q_{v}|W)$.
\item If $q-2$ vertices of $W$ are chosen from  $A$ and $B$ (selecting at least one vertex from each set), then
from Remark \ref{last}, there will exist at least two vertices of $Q$ that are adjacent only to vertices of $C$, say $q_{u}$ and $q_{v}$, and those two have same metric coordinates, i.e., $\delta(q_{u}|W) =\delta(q_{v}|W)$.
\item If $q-2$ vertices of $W$ are chosen from $A$ and $C$, then
from Remark \ref{last}, there will exist at least two vertices of $Q$ that are adjacent only to vertices of $B$, say $q_{u}$ and $q_{v}$, and those two have same metric coordinates, i.e., $\delta(q_{u}|W) =\delta(q_{v}|W)$.
\item  \label{item1} If $q-2$ vertices of $W$ are chosen from \( A \) and \( Q \); then, there will exist at least two vertices in \( A \), denoted as \( a_u \) and \( a_v \), where \( 1 \leq u, v\leq q-1\) and \(u\neq v \), and these two vertices are not adjacent to any vertex of  \( W \) and do not belong to set \( W \). This is because from the fact that \( |A| = q-1 \), and the neighbourhoods of distinct vertices in \( A \) are disjoint, i.e., \( N[a_r] \cap N[a_s] = \emptyset \) for \( 1 \leq r, s \leq q-1 \) with \( r \neq s \), and also one vertex is choosen from \( Q \), and all vertices in \( A \) are equidistant from every vertex in \( Q \), it follows (as stated in Remark \ref{22}) that  \( a_u \) and \( a_v \) have the same metric coordinates with respect to \( W \), i.e., $\delta(a_u | W) = \delta(a_v | W)$.
\item If $q-2$ vertices of $W$ are chosen from  $B$ and $C$, then we have the following cases: 
\begin{itemize}
    \item If at least two sets, say \(B^{\nu_1}\) and \(B^{\nu_2}\) (or \(C^{\nu_1}\) and \(C^{\nu_2}\), or \(B^{\nu_1}\) and \(C^{\nu_3}\), where \(1 \leq \nu_1, \nu_2, \nu_3 \leq \frac{p-1}{2}\), and \(\nu_1 \neq \nu_2\)), have exactly one vertex of \(W\), while the remaining sets ($B^{\nu_3}$, $C^{\nu_4}$, where \(1\leq \nu_4, \nu_3 \leq \frac{p-1}{2}\), and $\nu_1 \neq \nu_3 \neq \nu_2$) have either more than one vertex or do not have any vertices of \(W\), then we encounter the following contradictions:
    
(without loss of generality, consider the sets \(B^{\nu_1}\) and \(B^{\nu_2}\) as having exactly one vertex of \(W\), say \(b^{\nu_1}_i\) and \(b^{\nu_2}_j\), respectively, where \(1 \leq i, j \leq q-1\).)
    \begin{enumerate}
        \item If each $a_k$ ($1\leq k \leq q-1$) adjacent to at most one vertex of $W$, then $\delta( b^{\nu_2}_i|W) =\delta( b^{\nu_1}_{j}|W)$.
\item If at least one vertex in $A$, say $a_k$ ($1\leq k \leq q-1$) is adjacent to more than one vertex of $W$, then there will exist at least two vertices $a_u$ and $a_v$ in $A$, where $1 \leq u \neq v \leq q-1$, that are not adjacent to any vertex of $W$, and do not belong to set $W$. This is because $|A|=q-1$, and $N[a_r]\cap N[a_s]=\emptyset$, where $1\leq r,s\leq q-1$, $r\neq s$. According to Remark \ref{22}, these two vertices $a_u$ and $a_v$ have the same metric coordinates with respect to $W$, i.e., $\delta(a_u|W) = \delta(a_v|W)$.
 \end{enumerate}
    \item If only one set, say $B^{\nu_1}$ (or $C^{\nu_2}$), where $1 \leq \nu_1, \nu_2 \leq \frac{p-1}{2}$, has one vertex of $W$, say $b^{\nu_1}_i$, while the remaining sets have either more than one vertex of $W$ or do not have any vertices of $W$, then we encounter the following contradictions:
    \begin{enumerate}
        \item If each $a_k$ ($1\leq k \leq q-1$)  adjacent to at most one vertex of $W$, then $\delta(a_i|W) = \delta(q_j|W)$, where $q_j$ and $a_i$ is adjacent to $b^{\nu_1}_i$.
        \item If at least one vertex in $A$, say $a_k$ ($1\leq k \leq q-1$) is adjacent to more than one vertex of $W$, then there must be at least two vertices $a_u$ and $a_v$ in set $A$ (where $1 \leq u \neq v \leq q-1$) that are not adjacent to any vertex of $W$ and do not belong to set $W$. This is because $|A|=q-1$, and $N[a_r]\cap N[a_s]=\emptyset$, where $1\leq r,s\leq q-1$, $r\neq s$. According to Remark \ref{22}, these two vertices $a_u$ and $a_v$ have the same metric coordinates with respect to $W$, i.e., $\delta(a_u|W) = \delta(a_v|W)$.
        \end{enumerate}
     \item If all the sets $B^{\nu_1}$ and $C^{\nu_2}$ (where $1\leq \nu_1,\nu_2 \leq \frac{p-1}{2})$ either contain more than one vertex of  $W$ or do not contain any vertex of $W$, then there exist at least two sets do not contain any vertex of $W$, which we will denote by $B^{\nu_3}$ and $B^{\nu_4}$. This is because $W$ has $q-2$ vertices, and $q-2<2p-4$. Therefore, From Remark \ref{last} $\delta(q_i|W) = \delta(q_j|W)$, where $q_i$ and $q_j$ are adjacent to vertices of $B^{\nu_3}$ and $B^{\nu_4}$ respectively.  
   \end{itemize}
\item If $q-2$ vertices of $W$ are chosen from $B$ and $Q$ (selecting at least one vertex from each set), then $\delta(a_i|W) = \delta(a_j|W)$, where $a_i$ and $a_j$ are not adjacent to any vertices of $W$. The reason is the same as in the case \ref{item1}.
 \item If $q-2$ vertices of $W$ are chosen from $C$ and $Q$ (selecting at least one vertex from each set), then $\delta(a_i|W) = \delta(a_j|W)$, where $a_i$ and $a_j$ are not adjacent to any vertices of $W$. The reason is the same as in the previous case \ref{item1}.
 \item If $q-2$ vertices of $W$ are chosen from $A$, $B$, and $C$, then we have the following cases
 \begin{itemize}
     \item  If we choose at least two vertices from $A$, say $a_{k_1}$ and $a_{k_2}$, then
     \begin{enumerate}
         \item If at least two sets, say $B^{\nu_1}$ and $B^{\nu_2}$ (or $C^{\nu_1}$ and $C^{\nu_2}$, or $B^{\nu_1}$ and $C^{\nu_3}$), where $1\leq \nu_1, \nu_2, \nu_3 \leq \frac{p-1}{2}, \nu_1\neq \nu_2$, have exactly one vertex of $W$, say $b^{\nu_1}_i$ and $b^{\nu_2}_j$ ($1\leq i,j \leq q-1)$, respectively and remaining sets have more than one or do not have any vertices of $W$, then we encounter the following contradiction: 
         \begin{enumerate}
        \item If each $a_k$ ($1\leq k \leq q-1$) is adjacent to at most one vertex of $W$ except $a_{k_1}$ and $a_{k_2}$, and $a_{k_1}$ and $a_{k_2}$ not adjacent to any vertex of $W$, then $\delta( b^{\nu_2}_i|W) =\delta( b^{\nu_1}_{j}|W)$.
       \item If at least one vertex in $A$, say $a_k$ ($1\leq k \leq q-1$) is adjacent to more than one vertex of $W$, or If at least one vertex in $A$ say $a_k$ ($1\leq k \leq q-1$) is belonging to set $W$, and adjacent to at least one vertex of $W$, then there will exist at least two vertices in $A$ , say $a_u $ and $a_v $, where $1\leq u \neq v \leq q-1$, not adjacent to any vertex of $W$, and also not belonging to set $W$. This is because $|A|=q-1$, and $N[a_r]\cap N[a_s]=\emptyset$, where $1\leq r,s\leq q-1$, $r\neq s$. From Remark \ref{22} those two vertices $a_u $ and $a_v $ have same metric coordinates with respect to $W$, i.e.,      
 $\delta(a_{u}|W) =\delta(a_{v}|W)$.
    \end{enumerate}
    \item \label{item 9} If only one set, say \( B^{\nu_1} \) (or \( C^{\nu_2} \)), where \( 1 \leq \nu_1, \nu_2 \leq \frac{p-1}{2} \), contains one vertex of  \( W \), say \( b^{\nu_1}_i \), while the remaining sets contain either more than one or do not have any vertices of  \( W \), then there exist at least two sets do not contain any vertex of $W$, which we will denote by $B^{\nu_3}$ and $B^{\nu_4}$ (where $\nu_3 \neq \nu_4$). This is because $W$ has $q-2$ vertices, and $q-2<2p-4$, and also 2 vertices of $W$ are chosen in $A$. Therefore, From Remark \ref{last} $\delta(q_i|W) = \delta(q_j|W)$, where $q_i$ and $q_j$ are adjacent to vertices of $B^{\nu_3}$ and $B^{\nu_4}$ respectively.
 \item If all sets $B^{\nu_i}$ and $C^{\nu_j}$ (where $1\leq i,j \leq \frac{p-1}{2})$ either contain more than one vertex of $W$ or do not contain any vertex of $W$, then we arrive at the same contradiction mentioned in the previous case \ref{item 9}. 
   \end{enumerate} 
    \item If we choose 1 vertex from $A$, say $a_{k_1}$ ($1\leq k_1 \leq q-1$), and remaining vertex from $B$ and $C$, then we have the following cases
    \begin{enumerate}
        \item If at least two sets, say $B^{\nu_1}$ and $B^{\nu_2}$ (or $C^{\nu_1}$ and $C^{\nu_2}$, or $B^{\nu_1}$ and $C^{\nu_3}$, where $1\leq \nu_1,\nu_2,\nu_3 \leq \frac{p-1}{2}, \nu_1\neq \nu_2)$), have exactly one vertex of $W$, say $b^{\nu_1}_i$ and $b^{\nu_2}_j$ (where $1\leq i,j \leq \frac{p-1}{2})$, respectively and remaining sets have more than one or do not have any vertices of $W$, then we get the following contradictions:
         \begin{enumerate}
        \item If each $a_k$ ($1\leq k \leq q-1$) is adjacent to at most one vertex of $W$ except $a_{k_1}$, and $a_{k_1}$ is not adjacent to any vertex of $W$, then $\delta( b^{\nu_2}_i|W) =\delta( b^{\nu_1}_{j}|W)$.
        \item If at least one vertex in $A$, say $a_k$ ($1\leq k \leq q-1$) is adjacent to more than one vertex of $W$, or if one vertex in $A$, say $a_{k_1}$ ($1\leq k_1 \leq q-1$) is belonging to set $W$ and adjacent to at least one vertex of $W$, then there must be at least two vertices $a_u$ and $a_v$ in set $A$, where $1 \leq u \neq v \leq q-1$, that are not adjacent to any vertex of $W$ and do not belong to set $W$. This is because $|A|=q-1$, and $N[a_r]\cap N[a_s]=\emptyset$, where $1\leq r,s\leq q-1$, $r\neq s$. According to Remark \ref{22}, these two vertices $a_u$ and $a_v$ have the same metric coordinates with respect to $W$, i.e., $\delta(a_u|W) = \delta(a_v|W)$.
        \end{enumerate}
    \item  If only one set say, $B^{\nu_1}$  (or $C^{\nu_2}$), where $1\leq \nu_1,\nu_2 \leq \frac{p-1}{2}$, having one vertex of $W$, say $b^{\nu_1}_i$ and remaining sets have more than one or do not have any vertices of $W$, then we get the similar contradiction as mentioned in case \ref{item 9}
   \item If all sets $B^{\nu_1}$ and $C^{\nu_2}$ (Where $1\leq \nu_1,\nu_2 \leq \frac{p-1}{2})$ either contain more than one vertex of $W$ or do not contain any vertex of $W$, then we encounter the similar contradiction as mentioned in the case \ref{item 9}.
      \end{enumerate} 
 \end{itemize}
\item  If $q-2$ vertices of $W$ are chosen from $A$, $B$, and $Q$ (or from $A$, $C$, and $Q$, or from $B$, $C$, and $Q$, or from $A$, $B$, $C$, and $Q$ ), in any combination, then we get the same contradiction as mentioned in the case \ref{item1}.
\end{enumerate}
 It has been observed that when $p+1< q < 2p-2$, selecting a set (denoted as $W$) of $q-2$ vertices from $A$, $B$, $C$, and $Q$ in any combination, results in at least two vertices have the same metric coordinates with respect to $W$. This indicates that the MD of  $BS(\Gamma(R))$ is greater than $q-2$ when $p+1< q < 2p-2$.
 \end{proof}
 
\begin{theorem}
Let $p$ and $q$ be two distinct odd primes, where $q \geq 2p-1$. Suppose $R=\mathbb{Z}_n$ when $n=pq$, then,
$idim(BS(\Gamma(R))
)=q-2.$
\begin{proof}
Theorem $\ref{b}$ clearly states that $dim(BS(\Gamma(R)=q-2$, and the metric basis considered in Theorem $\ref{b}$ is an independent set. This concludes the theorem by Definition $\ref{defn1}$ and Definition \ref{defn2}.
\end{proof}
\end{theorem}

            \begin{theorem}\label{z}
Let $p$ and $q$ be two distinct odd primes, where $q>p$ $\&$ $q = 2p-3$. Suppose $R=\mathbb{Z}_n$, where $n=pq$, then,
$dim(BS(\Gamma(R))
)=q-1$.
\begin{proof}
Consider a ring $R = \mathbb{Z}_n$, where $n=pq$, where $q = 2p-3$. Let $\Gamma(R)$ denote the zero divisor graph of $R$. We denote $BS(\Gamma(R))$ as the barycentric subdivision of $\Gamma(R)$. The graph $BS(\Gamma(R))$ consists of $pq-1$ vertices and $2(p-1)(q-1)$ edges. Specifically, the vertex set of $BS(\Gamma(R))$ is partitioned into 4 sets: $A=\{a_1,a_2,\ldots, a_{q-1}\}$,  $B=\{b^\nu_i\}$, $C=\{c^\nu_i\}$, where $1\leq \nu\leq \frac{p-1}{2}$, and $1\leq i\leq {q-1}$ ($B^\nu=\{b^\nu_1,b^\nu_2,\ldots,b^\nu_{q-1}\}$, where $1\leq \nu \leq \frac{p-1}{2}$ are the subsets of $B$, and $C^\nu=\{c^\nu_1,c^\nu_2,\ldots,c^\nu_{q-1}\}$, where $1\leq \nu \leq \frac{p-1}{2}$  are the subsets of $C$), and $Q=\{q_1,q_2, \ldots, q_{p-1}\}$. These sets are illustrated in Figure \ref{fig:4}.\\

 \noindent Let us consider a set $E=\{   a _1, c^1_2,c^1_3,c^2_4,c^2_5,\ldots,c^\frac{p-1}{2}_{p-1},c^\frac{p-1}{2}_{p}, b^1_{p+1}, b^1_{p+2},\ldots, b^\frac{p-5}{2}_{2p-6},b^\frac{p-5}{2}_{2p-5},b^\frac{p-3}{2}_{2p-4},\ldots,b^\frac{p-3}{2}_{q-1}\}$. 
Then, to prove that $E$  is a resolving set of $BS(\Gamma(R))$, we need to assign unique metric coordinates for every vertex of $V(BS(\Gamma(R)))-E$ with respect to set $E$. \\\\
 The metric coordinates for the vertices $\{a_ \varrho:1\leq \noindent \varrho\leq q-1\}$ are given below
 \begin{equation*}
               \delta(a_{\varrho}|E)=
                   \begin{cases}
                   (4_{\{1\}},1_{\{ 2 \}},3_{\{3\}},\ldots,3_{\{q-2\}},3_{\{q-1\}}) &\textit{if}\quad \varrho =2  ;\\
                    (4_{\{1\}},3_{\{2\}},\ldots,3_{\{\varrho -1\}},1_{\{ \varrho \}},3_{\{\varrho +1\}},\ldots,3_{\{q-1\}}) &\textit{if}\quad 3\leq \varrho $ $\leq q-2  ;\\
                    (4_{\{1\}},3_{\{2\}},\ldots, 3_{\{q-2\}}),1_{\{q-1\}}) &\textit{if}\quad  \varrho = q-1  .\\
                    \end{cases}
                     \end{equation*}
                     
                    \noindent The metric coordinates for the vertices 
                     $\{b^\nu_\varrho:1\leq \nu\leq \frac{p-5}{2} $ $ \& $  $1\leq \varrho\leq q-1 \}$ are given below
                        \begin{equation*}
                        \delta(b^{\nu}_{\varrho}|E)=
                   \begin{cases}
                   (1_{\{1\}},4_{\{2\}},\ldots,4_{\{2\nu+p-2\}},2_{\{p+2\nu-1\}},2_{\{p+2\nu\}},4_{\{p+2\nu+1\}},\ldots,4_{\{q-1\}})\\\qquad  \qquad \qquad \qquad \qquad \qquad \qquad \qquad \qquad \qquad  \qquad \qquad\qquad\qquad\quad\textit{if}\quad \varrho= 1;\\
         (3_{\{1\}}, 2_{\{ 2\}},4_{\{3\}},\ldots,4_{\{2\nu+p-2\}},2_{\{p+2\nu-1\}},2_{\{p+2\nu\}},4_{\{p+2\nu+1\}},\ldots,4_{\{q-1\}})\\
         \qquad  \qquad \qquad \qquad \qquad \qquad \qquad \qquad \qquad \qquad  \qquad \qquad\qquad\qquad\quad\textit{if}\quad \varrho=2 ;\\
                     (3_{\{1\}},4_{\{2\}}, \ldots,4_{\{\varrho-1 \}},2_{\{ \varrho \}},4_{\{\varrho +1\}},\ldots,4_{\{2\nu+p-2\}},2_{\{p+2\nu-1\}},2_{\{p+2\nu\}},4_{\{p+2\nu+1\}},\ldots,4_{\{q-1\}})\\
                     \qquad  \qquad \qquad \qquad \qquad \qquad \qquad \qquad \qquad \qquad  \qquad \qquad\qquad\qquad\quad\textit{if}\quad  3 \leq  \varrho  \leq 2\nu+p-3 ;\\
(3_{\{1\}},4_{\{2\}}, \ldots,4_{\{2\nu \}},2_{\{ 2\nu+p-2 \}},2_{\{p+2\nu-1\}},2_{\{p+2\nu\}},4_{\{p+2\nu+1\}},\ldots,4_{\{q-1\}})\\
\qquad  \qquad \qquad \qquad \qquad \qquad \qquad \qquad \qquad \qquad  \qquad \qquad\qquad\qquad\quad\textit{if}\quad   \varrho=2\nu+p-2 ;\\
(3_{\{1\}},4_{\{2\}}, \ldots,4_{\{2\nu+p-2\}},2_{\{p+2\nu-1\}},2_{\{p+2\nu\}},2_{\{ p+2\nu+1 \}},4_{\{ p+2\nu+2 \}},\ldots,4_{\{q-1\}})\\
\qquad  \qquad \qquad \qquad \qquad \qquad \qquad \qquad \qquad \qquad  \qquad \qquad\qquad\qquad\quad\textit{if}\quad  \varrho=p+2\nu+1;\\
                    (3_{\{1\}},4_{\{2\}}, \ldots,4_{\{2\nu+p-2\}},2_{\{p+2\nu-1\}},2_{\{p+2\nu\}},4_{\{p+2\nu+1\}},\ldots,4_{\{\varrho-1 \}},2_{\{ \varrho \}},4_{\{\varrho +1\}},\ldots,4_{\{q-1\}})\\
\quad \qquad \qquad \qquad \qquad \qquad \qquad \qquad \qquad \qquad  \qquad \qquad\qquad\qquad\quad\textit{if}\quad p+2\nu+2 \leq  \varrho  \leq q-2;\\
                     (3_{\{1\}},4_{\{2\}}, \ldots,4_{\{2\nu+p-2\}},2_{\{p+2\nu-1\}},2_{\{p+2\nu\}},4_{\{p+2\nu+1\}},\ldots, 4_{\{q-2\}}, 2_{\{q-1\}})\\
                     \qquad  \qquad \qquad \qquad \qquad \qquad \qquad \qquad \qquad \qquad  \qquad \qquad\qquad\qquad\quad\textit{if}\quad  \varrho = q-1.\\
                     \end{cases}
                    \end{equation*}
                    \noindent The metric coordinates for the vertices 
                     $\{b^\nu_\varrho: \nu =\frac{p-3}{2} $ $\&$  $1\leq \varrho\leq q-1 \}$ are given below
                        \begin{equation*}
                        \delta(b^{\nu}_{\varrho}|E)=
                   \begin{cases}
                   (1_{\{1\}},4_{\{2\}}, \ldots, 4_{\{2p-5 \}},2_{\{2p-4 \}}, \ldots,2_{\{q-3\}},2_{\{q-1\}})\\
\qquad  \qquad \qquad \qquad \qquad \qquad \qquad \qquad \qquad \qquad  \qquad \qquad\qquad\qquad\quad\textit{if}\quad \varrho= 1;\\
                   (3_{\{1\}},2_{\{ 2\}},4_{\{3\}}, \ldots, 4_{\{2p-5 \}},2_{\{2p-4 \}}, \ldots,2_{\{q-3\}},2_{\{q-1\}})\\
\qquad  \qquad \qquad \qquad \qquad \qquad \qquad \qquad \qquad \qquad  \qquad \qquad\qquad\qquad\quad\textit{if}\quad    \varrho =2  ;\\
                   (3_{\{1\}},4_{\{2\}}, \ldots,4_{\{\varrho-1\}},2_{\{ \varrho\}},4_{\{\varrho +1\}}, \ldots, 4_{\{2p-5 \}},2_{\{2p-4 \}}, \ldots,2_{\{q-3\}},2_{\{q-1\}})\\
\qquad  \qquad \qquad \qquad \qquad \qquad \qquad \qquad \qquad \qquad  \qquad \qquad\qquad\qquad\quad\textit{if}\quad  3 \leq  \varrho  \leq 2p-6  ;\\
                   3_{\{1\}},4_{\{2\}}, \ldots,4_{\{2p-6\}},2_{\{ 2p-5\}},2_{\{2p-4 \}}, \ldots,2_{\{q-3\}},2_{\{q-1\}})\\
\qquad  \qquad \qquad \qquad \qquad \qquad \qquad \qquad \qquad \qquad  \qquad \qquad\qquad\qquad\quad\textit{if}\quad  \varrho = 2p-5.\\
                    \end{cases}
                    \end{equation*}
\noindent The metric coordinates for the vertices 
                     $\{b^\nu_\varrho: \nu =\frac{p-1}{2} $ $\&$  $1\leq \varrho\leq q-1 \}$ are given below
                        \begin{equation*}
                        \delta(b^{\nu}_{\varrho}|E)=
                   \begin{cases}
                   (1_{\{1\}},4_{\{2\}},\ldots, 4_{\{2p-3 \}},4_{\{2p-2 \}}, \ldots,4_{\{q-3\}},4_{\{q-1\}})\\
\qquad  \qquad \qquad \qquad \qquad \qquad \qquad \qquad \qquad \qquad  \qquad \qquad\qquad\qquad\quad\textit{if}\quad\varrho= 1;\\
                   (3_{\{1\}},2_{\{ 2\}},4_{\{3\}}, \ldots,4_{\{q-1\}})\\
\qquad  \qquad \qquad \qquad \qquad \qquad \qquad \qquad \qquad \qquad  \qquad \qquad\qquad\qquad\quad\textit{if}\quad    \varrho =2;\\
                   (3_{\{1\}},4_{\{2\}}, \ldots,4_{\{\varrho-1\}},2_{\{ \varrho\}},4_{\{\varrho +1\}}, \ldots,4_{\{q-1\}})\\
\qquad  \qquad \qquad \qquad \qquad \qquad \qquad \qquad \qquad \qquad  \qquad \qquad\qquad\qquad\quad\textit{if}\quad  3 \leq  \varrho  \leq q-2  ;\\
                      (3_{\{1\}},4_{\{2\}},\ldots,4_{\{q-2\}},2_{\{q-1\}})\\
\qquad  \qquad \qquad \qquad \qquad \qquad \qquad \qquad \qquad \qquad  \qquad \qquad\qquad\qquad\quad\textit{if}\quad   \varrho= q-1  .\\
                    \end{cases}
                    \end{equation*}

\noindent The metric coordinates for the vertices 
                     $\{c^1_\varrho: 1\leq \varrho\leq q-1 \}$ are given below
                        \begin{equation*}
                        \delta(c^{1}_{\varrho}|E)=
                   \begin{cases}
                   (1_{\{1\}},2_{\{2\}},2_{\{3\}},4_{\{4\}},\ldots,4_{\{q-1\}})&\textit{if}\quad\varrho= 1;\\
                     (3_{\{1\}},2_{\{2\}}, 2_{\{3 \}},2_{\{ 4 \}}, 4_{\{ 5 \}},\ldots,4_{\{q-1\}})&\textit{if}\quad \varrho= 4;\\
                     (3_{\{1\}},2_{\{2\}}, 2_{\{3 \}},4_{\{ 4 \}}, \ldots, 4_{\{\varrho-1 \}},2_{\{\varrho \}}, 4_{\{\varrho+1 \}},\ldots,4_{\{q-1\}})&\textit{if}\quad  5 \leq  \varrho  \leq q-2  ;\\
                     (3_{\{1\}},2_{\{2\}},2_{\{3\}} 4_{\{4\}},\ldots, 4_{\{q-2\}},2_{\{q-1\}}) &\textit{if}\quad\varrho= q-1.\\
                \end{cases}
                    \end{equation*}

                    \noindent The metric coordinates for the vertices 
                     $\{c^2_\varrho: 1\leq \varrho\leq q-1 \}$ are given below
                        \begin{equation*}
                        \delta(c^{2}_{\varrho}|E)=
                   \begin{cases}
                   (1_{\{1\}},4_{\{2\}},4_{\{3\}},2_{\{4\}}, 2_{\{5\}}, 4_{\{6\}}, \ldots,4_{\{q-1\}})&\textit{if}\quad\varrho= 1;\\
                   (3_{\{1\}},2_{\{2\}},4_{\{3\}},2_{\{4\}}, 2_{\{5\}}, 4_{\{6\}}, \ldots,4_{\{q-1\}})&\textit{if}\quad\varrho= 2;\\
                     (3_{\{1\}},4_{\{2\}},2_{\{3\}},2_{\{4\}}, 2_{\{5\}}, 4_{\{6\}}, \ldots,4_{\{q-1\}})&\textit{if}\quad\varrho= 3;\\
                     (3_{\{1\}},4_{\{2\}},4_{\{3\}},2_{\{4\}}, 2_{\{5\}}, 2_{\{6\}}, 4_{\{7\}}  \ldots,4_{\{q-1\}})&\textit{if}\quad\varrho= 6;\\
                     
                   (3_{\{1\}},4_{\{2\}}, 4_{\{3 \}},2_{\{ 4 \}}, 2_{\{ 5 \}}, 4_{\{ 6 \}}, \ldots, 4_{\{\varrho-1 \}}2_{\{\varrho \}}, 4_{\{\varrho+1 \}},\ldots,4_{\{q-1\}})&\textit{if}\quad  7 \leq  \varrho  \leq q-2  ;\\
                     (3_{\{1\}},4_{\{2\}},4_{\{3\}},2_{\{4\}}, 2_{\{5\}}, 4_{\{6\}}, \ldots 4_{\{q-2\}} ,2_{\{q-1\}})&\textit{if}\quad\varrho= q-1.\\
                    \end{cases}
                    \end{equation*}
                    
                     \noindent The metric coordinates for the vertices 
                     $\{c^\nu_\varrho:3\leq \nu\leq \frac{p-1}{2}$  $\&$  $1\leq \varrho\leq q-1 \}$ are given below
                        \begin{equation*}
                        \delta(c^{\nu}_{\varrho}|E)=
                   \begin{cases}
                   (1_{\{1\}},4_{\{2\}},\ldots, 4_{\{2\nu-1\}},2_{\{2\nu\}},2_{\{2\nu+1\}},4_{\{2\nu+2\}},\ldots,4_{\{q-1\}})\\
\qquad  \qquad \qquad \qquad \qquad \qquad \qquad \qquad \qquad \qquad  \qquad \qquad\qquad\qquad\quad \textit{if}\quad \varrho= 1;\\
                     (3_{\{1\}},2_{\{ 2 \}},4_{\{3\}}, \ldots, 4_{\{2\nu-1\}},2_{\{2\nu\}},2_{\{2\nu+1\}},4_{\{2\nu+2\}},\ldots,4_{\{q-1\}})\\
\qquad  \qquad \qquad \qquad \qquad \qquad \qquad \qquad \qquad \qquad  \qquad \qquad\qquad\qquad\quad \textit{if}\quad\varrho= 2;\\
                   (3_{\{1\}},4_{\{2\}}, \ldots,4_{\{\varrho-1 \}},2_{\{ \varrho \}},4_{\{\varrho +1\}}, \ldots, 4_{\{2\nu-1\}},2_{\{2\nu\}},2_{\{2\nu+1\}},4_{\{2\nu+2\}},\ldots,4_{\{q-1\}})\\
                   \qquad  \qquad \qquad \qquad \qquad \qquad \qquad \qquad \qquad \qquad  \qquad \qquad\qquad\qquad\quad \textit{if}\quad  3 \leq  \varrho  \leq 2\nu-2  ;\\
                      (3_{\{1\}},4_{\{2\}}, \ldots,4_{\{2\nu-2 \}},2_{\{ 2\nu-1 \}},2_{\{2\nu\}},2_{\{2\nu+1\}},4_{\{2\nu+2\}},\ldots,4_{\{q-1\}})\\
\qquad  \qquad \qquad \qquad \qquad \qquad \qquad \qquad \qquad \qquad  \qquad \qquad\qquad\qquad\quad\textit{if}\quad \varrho= 2\nu-1 ;\\
                     (3_{\{1\}},4_{\{2\}}, \ldots,4_{\{2\nu-1\}},2_{\{2\nu\}},2_{\{2\nu+1\}},2_{\{ 2\nu+2 \}}, 4_{\{2\nu+3 \}},\ldots,4_{\{q-1\}})\\
\qquad  \qquad \qquad \qquad \qquad \qquad \qquad \qquad \qquad \qquad  \qquad \qquad\qquad\qquad\quad \textit{if}\quad \varrho= 2\nu+2 ;\\
                    (3_{\{1\}},4_{\{2\}}, \ldots,4_{\{2\nu-1\}},2_{\{2\nu\}},2_{\{2\nu+1\}},4_{\{2\nu+2\}},\ldots,4_{\{\varrho-1 \}},2_{\{ \varrho \}},4_{\{\varrho +1\}},\ldots,4_{\{q-1\}})\\
\qquad  \qquad \qquad \qquad \qquad \qquad \qquad \qquad \qquad \qquad  \qquad \qquad\qquad\qquad\quad \textit{if}\quad 2\nu+3 \leq  \varrho  \leq q-2  ;\\
                    (3_{\{1\}},4_{\{2\}}, \ldots,4_{\{2\nu-1\}},2_{\{2\nu\}},2_{\{2\nu+1\}},4_{\{2\nu+2\}},\ldots, 4_{\{q-2\}}),2_{\{q-1\}})\\\qquad  \qquad \qquad \qquad \qquad \qquad \qquad \qquad \qquad \qquad  \qquad \qquad\qquad\qquad\quad \textit{if}\quad\varrho= q-1.\\
                    \end{cases}
                    \end{equation*}
                    
                \noindent The metric coordinates for the vertices 
                     $\{q_\nu:1\leq \nu\leq {p-1} \}$ are given below
                        \begin{equation*}
                        \delta(q_{\nu}|E)=
                   \begin{cases}
(2_{\{1\}},1_{\{2\}}, 1_{\{3\}}, 3_{\{4\}},\ldots,3_{\{q-1\}})&\textit{if}\quad  \nu=1 ;\\
                   (2_{\{1\}},3_{\{2\}}, \ldots, 3_{\{2\nu-1\}},1_{\{2\nu\}},1_{\{2\nu+1\}},3_{\{2\nu+2\}},\ldots,3_{\{q-1\}})&\textit{if}\quad  2 \leq \nu  \leq p-3 ;\\
                     (2_{\{1\}},3_{\{2\}}, \ldots, 3_{\{2\nu-1\}},1_{\{2\nu\}},\ldots,1_{\{q-1\}})&\textit{if}\quad  \nu= p-2 ;\\
                     (2_{\{1\}},3_{\{2\}}, \ldots, 3_{\{q-1\}})&\textit{if}\quad  \nu= p-1.\\
                     \end{cases}
                     \end{equation*}
                     
\noindent Based on the listed metric coordinates, all vertices of $BS(\Gamma(R))$ have unique metric coordinates with respect to $E$. It follows that the $dim(BS(\Gamma(R)))\leq q-1$.
We need to show that $dim(BS(\Gamma(R)))\geq q-1$. From the Theorem \ref{b} it is evident that dim($BS(\Gamma(R)) > q-2$ ($\because q=2p-3$). This indicates that dim($BS(\Gamma(R)) = q-1$. 
\end{proof}
\begin{example}
For $p=7$ $\&$ $q=11$ \\
\noindent Let us consider a set $E=\{ a_1, c^1_2,c^1_3,c^2_4,c^2_5, c^3_6, c^3_7, b^1_8,b^1_9, b^2_{10}\}$,  
then, to prove that $E$  is a resolving set of $BS(\Gamma(\mathbb{Z}_{77}))$, we need to assign unique metric coordinates for every vertex of $BS(\Gamma(\mathbb{Z}_{77}))-E$ with respect to set $E$. These sets are illustrated in Figure \ref{fig:11}. 
\begin{figure}[ht!]
    \centering
    \includegraphics[width=11cm]{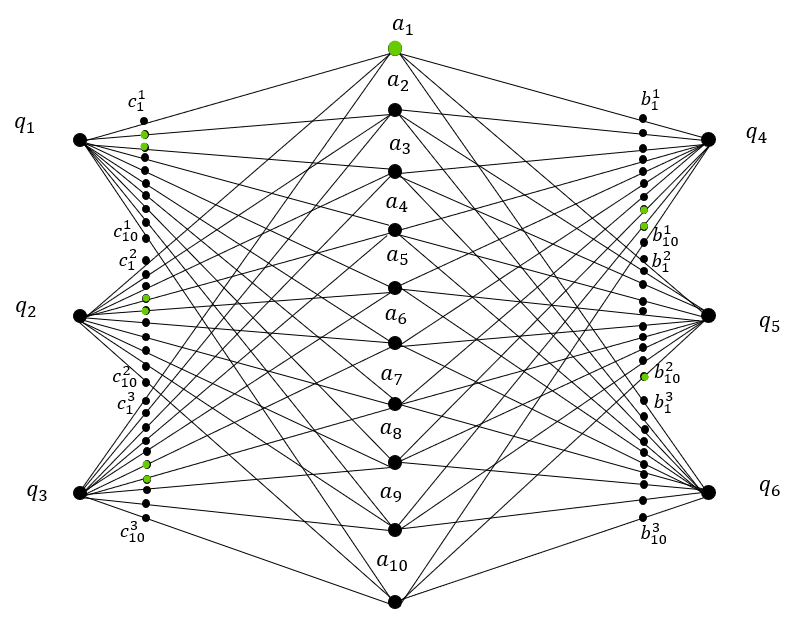}
    \caption{Barycentric Subdivision of Zero Divisor Graph of $\mathbb{Z}_{77}$}
    \label{fig:11}
\end{figure}

Metric coordinates of each vertex of $BS(\Gamma(\mathbb{Z}_{77}))-E$ with respect to set $E$ are given below
\renewcommand{\arraystretch}{1.3} 
\newpage
\begin{table}[h!]
    \begin{tabular}{|p{0.4\linewidth}|p{0.4\linewidth}|p{0.4\linewidth}|}
      \hline
        $\delta(b^1_{1}|E)=(1,4,4,4,4,4,4,2,2,4)$ &  $\delta(b^2 _1|E)=(1,4,4,4,4,4,4,4,4,2)$  \\ 
       
        $\delta(b^1_{2}|E)= (3,2,4,4,4,4,4,2,2,4)$ &  $\delta(b^2 _2|E)=(3,2,4,4,4,4,4,4,4,2)$     \\ 
       
        $\delta(b^1_{3}|E)= (3,4,2,4,4,4,4,2,2,4)$ &  $\delta(b^2 _3|E)=(3,4,2,4,4,4,4,4,4,2)$ \\
     
        $\delta(b^1_{4}|E)= (3,4,4,2,4,4,4,2,2,4)$ &  $\delta(b^2 _4|E)=(3,4,4,2,4,4,4,4,4,2)$   \\
       
        $\delta(b^1_{5}|E)= (3,4,4,4,2,4,4,2,2,4)$ & $\delta(b^2 _5|E)=(3,4,4,4,2,4,4,4,4,2)$    \\
    
      $\delta(b^1_{6}|E)= (3,4,4,4,4,2,4,2,2,4)$ &  $\delta(b^2 _6|E)=(3,4,4,4,4,2,4,4,4,2)$   \\

      $\delta(b^1_{7}|E)= (3,4,4,4,4,4,2,2,2,4)$ &  $\delta(b^2 _7|E)=(3,4,4,4,4,4,2,4,4,2)$    \\

      $\delta(b^1_{8}|E)= (3,4,4,4,4,4,4,0,2,4)$&   $\delta(b^2 _8|E)=(3,4,4,4,4,4,4,2,4,2)$     \\

      $\delta(b^1_{9}|E)= (3,4,4,4,4,4,4,2,0,4)$ &  $\delta(b^2 _9|E)=(3,4,4,4,4,4,4,4,2,2)$ \\

      $\delta(b^1_{10}|E)= $$(3,4,4,4,4,4,4,2,2,2)$&   $\delta(b^2 _{10}|E)=(3,4,4,4,4,4,4,4,4,0)$      \\
      \hline
      
    \end{tabular}
    
    \end{table}
    \begin{table}[h!]
    \begin{tabular}
    {|p{0.4\linewidth}|p{0.4\linewidth}|p{0.4\linewidth}|}
    \hline
      $\delta(b^3 _1|E)=(1,4,4,4,4,4,4,4,4,4)$&         $\delta(a^1 _1|E)=(0,3,3,3,3,3,3,3,3,3)$\\
      $\delta(b^3 _2|E)=(3,2,4,4,4,4,4,4,4,4)$&$\delta(a^1 _2|E)= (4,1,3,3,3,3,3,3,3,3) $\\
      
        $\delta(b^3 _3|E)=(3,4,2,4,4,4,4,4,4,4)$&$\delta(a^1 _3|E)= (4,3,1,3,3,3,3,3,3,3)$\\
        $\delta(b^3 _4|E)=(3,4,4,2,4,4,4,4,4,4)$&$\delta(a^1 _4|E)= (4,3,3,1,3,3,3,3,3,3)$\\
        $\delta(b^3 _5|E)=(3,4,4,4,2,4,4,4,4,4)$&$\delta(a^1 _5|E)= (4,3,3,3,1,3,3,3,3,3)$\\
         $\delta(b^3 _6|E)=(3,4,4,4,4,2,4,4,4,4)$& $\delta(a^1 _6|E)= (4,3,3,3,3,1,3,3,3,3)$ \\
          $\delta(b^3 _{7}|E)=(3,4,4,4,4,4,2,4,4,4)$&$\delta(a^1 _7|E)= (4,3,3,3,3,3,1,3,3,3)$ \\
          $\delta(b^3_{8}|E)=(3,4,4,4,4,4,4,2,4,4)$ &$\delta(a^1 _8|E)= (4,3,3,3,3,3,3,1,3,3)$\\
           $\delta(b^3_{9}|E)=(3,4,4,4,4,4,4,4,2,4)$& $\delta(a^1 _9|E)= (4,3,3,3,3,3,3,3,1,3)$\\
           $\delta(b^3_{10}|E)=(3,4,4,4,4,4,4,4,4,2)$& $\delta(a^1 _{10}|E)= (4,3,3,3,3,3,3,3,3,1)$\\
           \hline
    \end{tabular}
    
    \end{table}
 \begin{table}[h!]
\begin{tabular}{|p{0.4\linewidth}|p{0.4\linewidth}|p{0.4\linewidth}|}
      \hline
        $\delta(c^1 _1|E)=(1,2,2,4,4,4,4,4,4,4)$ &  $\delta(c^2 _1|E)=(1,4,4,2,2,4,4,4,4,4)$    \\ 
       $\delta(c^1 _2|E)= (3,0,2,4,4,4,4,4,4,4)$ &  $\delta(c^2 _2|E)=(3,2,4,2,2,4,4,4,4,4)$    \\ 
       $\delta(c^1 _3|E)= (3,2,0,4,4,4,4,4,4,4)$ &  $\delta(c^2 _3|E)=(3,4,2,2,2,4,4,4,4,4)$   \\
     $\delta(c^1 _4|E)= (3,2,2,2,4,4,4,4,4,4)$ &  $\delta(c^2 _4|E)=(3,4,4,0,2,4,4,4,4,4)$   \\
       $\delta(c^1 _5|E)= (3,2,2,4,2,4,4,4,4,4)$ & $\delta(c^2 _5|E)=(3,4,4,2,0,4,4,4,4,4)$    \\
    $\delta(c^1 _6|E)= (3,2,2,4,4,2,4,4,4,4)$ &  $\delta(c^2 _6|E)=(3,4,4,2,2,2,4,4,4,4)$   \\
$\delta(c^1 _7|E)= (3,2,2,4,4,4,2,4,4,4)$ &  $\delta(c^2 _7|E)=(3,4,4,2,2,4,2,4,4,4)$    \\
$\delta(c^1 _8|E)= (3,2,2,4,4,4,4,2,4,4)$ &  $\delta(c^2 _8|E)=(3,4,4,2,2,4,4,2,4,4)$     \\
$\delta(c^1 _9|E)= (3,2,2,4,4,4,4,4,2,4)$ &  $\delta(c^2 _9|E)=(3,4,4,2,2,4,4,4,2,4)$  \\
$\delta(c^1 _{10}|E)= $$(3,2,2,4,4,4,4,4,4,2)$ &  $\delta(c^2 _{10}|E)=(3,4,4,2,2,4,4,4,4,2)$     \\
      \hline
    \end{tabular}
\end{table}
\newpage
\begin{table}[h!]
    \begin{tabular}{|p{0.4\linewidth}|p{0.4\linewidth}|}
      \hline
 $\delta(c^3 _1|E)=(1,4,4,4,4,2,2,4,4,4)$& $ \delta(q_1|E)=(2,1,1,3,3,3,3,3,3,3)$\\
 $\delta(c^3 _2|E)=(3,2,4,4,4,2,2,4,4,4)$  & $ \delta(q_2|E)=(2,3,3,1,1,3,3,3,3,3)$\\
  $\delta(c^3 _3|E)=(3,4,2,4,4,2,2,4,4,4)$ & $ \delta(q_3|E)=(2,3,3,3,3,1,1,3,3,3)$\\
    $\delta(c^3 _4|E)=(3,4,4,2,4,2,2,4,4,4)$ &$ \delta(q_4|E)=(2,3,3,3,3,3,3,1,1,3)$\\
     $\delta(c^3 _5|E)=(3,4,4,4,2,2,2,4,4,4)$ &$ \delta(q_5|E)=(2,3,3,3,3,3,3,3,3,1)$\\
    $\delta(c^3 _6|E)=(3,4,4,4,4,0,2,4,4,4)$  & $ \delta(q_6|E)=(2,3,3,3,3,3,3,3,3,3)$\\
     $\delta(c^3_7|E)=(3,4,4,4,4,2,0,4,4,4)$  & \\
      $\delta(c^3_8|E)=(3,4,4,4,4,2,2,2,4,4)$&  \\
      $\delta(c^3_9|E)=(3,4,4,4,4,2,2,4,2,4)$&  \\
     $\delta(c^3_{10}|E)=(3,4,4,4,4,2,2,4,4,2)$ &  \\
      \hline
    \end{tabular}
    \label{tab:example3}
\end{table}
\end{example}
\noindent Based on the observation above, all vertices of $BS(\Gamma(\mathbb{Z}_{77}))$ have unique metric coordinates with respect to $E$. So $dim(BS(\Gamma(\mathbb{Z}_{77})))$ is less than or equal to 10. From Theorem \ref{b}, it is evident that $dim(BS(\Gamma(\mathbb{Z}_{77})))\geq 10$ (since $q=2p-3$). This indicates that $dim(BS(\Gamma(\mathbb{Z}_{77})))=10$.\\
 \end{theorem}
 \begin{theorem}
Let $p$ and $q$ be two distinct odd primes, where $q>p$ $\&$ $q = 2p-3$. Suppose $R=\mathbb{Z}_n$, where $n=pq$, then
$idim(BS(\Gamma(R))
)=q-1$.
\begin{proof}
Theorem $\ref{z}$ clearly states that $dim(BS(\Gamma(R))=q-1$, and the metric basis considered in Theorem $\ref{z}$ is an independent set. This concludes the theorem by Definition $\ref{defn1}$ and Definition \ref{defn2}.
\end{proof}
\end{theorem}
\section {Conclusion}
 \noindent This study aimed to calculate the MD of the barycentric subdivision of the zero divisor graph $\mathbb {Z}_{pq}$, where $p$ and $q$ are distinct primes with $q$ greater than $p$. MD of the barycentric subdivision is an important graph transformation with various applications in network design, computational geometry, and topological data analysis. In this article, we proved that the MD of the barycentric subdivision of $\Gamma(\mathbb{Z}_{pq})$ is greater than or equal to $q-2$ (where $p\geq 2$, and $q\geq 2p-1$). And also investigating the MD of the barycentric subdivision of $\Gamma(\mathbb{Z}_{pq})$, where $p< q \leq 2p-5$ will be an exciting topic for future research.\\

 \section*{Data Availability}
Data sharing does not apply to this manuscript because no data sets were analyzed or generated during this particular study.
\section*{Conflicts of Interest}
Authors have nothing to declare as a conflict of interest.
\section*{Authors’ Contributions}
All authors contributed equally.

  	\bibliographystyle{tfcad}
\bibliography{mybib}
 \end{document}